\documentclass[11pt,reqno]{amsart}
\usepackage{amsmath,amssymb,amsthm, hyperref, amscd}
\usepackage[T1]{fontenc}
\usepackage{textcomp}
\usepackage[scale=0.95,ttdefault]{AnonymousPro}
\usepackage[noadjust]{cite}

\hypersetup{
    colorlinks=true,
    linkcolor=blue,
    urlcolor=cyan,
    }
\title{On some properties of birational derived splinters} % and rational singularities}
\newcommand{\fa}{\mathfrak{a}}

\newcommand{\fm}{\mathfrak{m}}
\newcommand{\fn}{\mathfrak{n}}
\newcommand{\fM}{\mathfrak{M}}

\newcommand{\fp}{\mathfrak{p}}
\newcommand{\fq}{\mathfrak{q}}
\newcommand{\ft}{\mathfrak{t}}

\newcommand{\Spec}{\operatorname{Spec}}

\newcommand{\red}{\operatorname{red}}
\newcommand{\Hom}{\operatorname{Hom}}

\newcommand{\colim}{\operatorname{colim}}
\newcommand{\cO}{\mathcal{O}}
\newcommand{\cU}{\mathcal{U}}

\newcommand{\bF}{\mathbf{F}}

\newcommand{\bQ}{\mathbf{Q}}

\newcommand{\bZ}{\mathbf{Z}}

\newcommand{\Tag}[1]{\href{https://stacks.math.columbia.edu/tag/#1}{\texttt{#1}}}
\newcommand{\citestacks}[1]{\cite[Tag \Tag{#1}]{Stacks}}
\newcommand{\citetwostacks}[2]{\cite[Tags \Tag{#1} and \Tag{#2}]{Stacks}}

\author{Shiji Lyu}
\newtheorem{Thm}{Theorem}[section]

\newtheorem{Lem}[Thm]{Lemma}
\newtheorem{Cor}[Thm]{Corollary}
\newtheorem{Prop}[Thm]{Proposition}

\theoremstyle{definition}
\newtheorem{Def}[Thm]{Definition}

\theoremstyle{remark}
\newtheorem{Rem}[Thm]{Remark}
\newtheorem{Ques}[Thm]{Question}

\begin{document}
\begin{abstract}
   A Noetherian reduced ring $A$ is called a birational derived splinter if for all proper birational maps $X\to\Spec(A)$,
   the canonical map $A\to Rf_*\cO_X$ splits.
   In equal characteristic zero this property characterizes rational singularities, but much less can be said in positive or mixed characteristics.
   In this paper, we prove some fundamental properties of this notion, including the behavior under localization, taking a pure subring, taking direct limit, and along an \'etale extension.
   In particular, direct limit of rational singularities in characteristic zero has rational singularities.
   Then, we study residue extensions (in arbitrary characteristic), and openness and regular extensions in positive characteristic, parallel to Datta-Tucker and the author's previous works on splinters.
\end{abstract}
\maketitle

\section{Introduction}
%To state the main results in the best generality available to us,
%we need the notion of a G-ring, see \citestacks{07GH}.
%A ring essentially of finite type over a Noetherian complete local ring (in particular a field) is a G-ring, see
%\citetwostacks{07PS}{07PV}.
%A Noetherian ring is a G-ring if and only if it is locally quasi-excellent, cf. \cite[(34.A)]{Matsumura}.

%\begin{Thm}[=Theorem \ref{thm:ScalarGMIRROR}]\label{thm:ScalarG}
%Let $(A,\fm_A,k_A)\to (B,\fm_B,k_B)$ be a flat homomorphism of Noetherian local rings.
%Assume that $\fm_B=\fm_AB$, that $k_B/k_A$ is separable,
%and that $A$ is a G-ring.
%If $A$ is a birational derived %splinter, so is $B$. % is a splinter.
%\end{Thm}
%n-excellent:
%We also obtain a version of Theorem \ref{thm:ScalarG} without the G-ring hypothesis, but it requires a regular base ring to apply base change.
%Note that in the statement below we do not assume $R\to S$ injective.
In algebraic geometry, it is extremely useful to consider proper birational models $Y\to X$ of a given algebraic variety $X$.
One question that naturally arises is how to compare the (coherent or local) cohomology of $X$ and that of $Y$.
Works in this direction include
\cite{CR15}.

In this paper, we shall specialize to the problem of local splitting.
More explicitly, we will study the (Noetherian reduced) rings $A$ such that for all proper birational maps $f:X\to\Spec(A)$,
the canonical map $A\to Rf_*\cO_X$
splits in the derived category of $A$.
We say $A$ is a \emph{birational derived splinter}.

There are two closely related notions, \emph{splinters} and \emph{derived splinters},
defined by considering finite surjective and proper surjective $f$'s respectively instead of proper birational $f$'s.
Regular rings are derived splinters \cite{Bha18},
and splinters and derived splinters are the same in positive and mixed characteristic (\cite{Bha12} and the forthcoming work \cite{BL}),
so in particular, splinters are birational derived splinters.
In characteristic zero, a result of Kova\'cs
\cite{KovacsOther}
shows that for affine varieties over the complex numbers,
the notion of birational derived splinter is the same as the notion of rational singularity.
Recently, Murayama \cite{Takumi} has extended the Kodaira-type vanishing theorems to general schemes of equal characteristic zero,
from which it follows that
the notion of birational derived splinter is the same as the notion of rational singularity
for all quasi-excellent rings of equal characteristic zero,
as noted below in Proposition \ref{prop:charactrizeQBDS}.

However, 
the notion of birational derived splinters does not seem to be very well studied itself,
especially in positive and mixed characteristics.
This is the aim of this paper.
Our first main result is a collection of fundamental properties of birational derived splinters
(Corollary \ref{cor:MDSisLocal}, Corollary \ref{cor:NoetherianPureDescendMDS}, Theorem \ref{thm:MDSlimitHARD}, and Theorem \ref{thm:EtaleDominateM}): %, in view of Lemma \ref{lem:characterizeNoetherianMDS}):
\begin{Thm}
\label{thm:NoetherianManifest}
Let $A$ be a Noetherian ring. 
Then the followings hold.

\begin{enumerate}
    \item $A$ is a birational derived splinter if and only if all localizations of $A$ at prime (resp. maximal) ideals of $A$ are birational derived splinters.
    
    \item
    If $A$ is a birational derived splinter, then
    for all cyclically pure ring maps $A'\to A$,
    $A'$ is a Noetherian birational derived splinter.
    
    \item \label{ManifestLimit} If $A$ is the direct limit of a system of Noetherian birational derived splinters,
    then $A$ is a birational derived splinter.
    
    \item If $A$ is a birational derived splinter and $B$ is an \'etale $A$-algebra, then $B$ is a birational derived splinter.
\end{enumerate}
\end{Thm}

We note that of these results, the most difficult one is (\ref{ManifestLimit}).
This is because of the lack of flatness, as mentioned in \cite[Remark 5.3.3(2)]{ADVal}.
Specializing to characteristic zero,
we have the following consequence, which the author believes to be new:

\begin{Cor}[=Corollary \ref{cor:limitQmirror}]\label{cor:MainLimitQ}
Let $A$ be a Noetherian quasi-excellent $\bQ$-algebra.
If $A$ is the direct limit of a system of Noetherian quasi-excellent $\bQ$-algebras that have rational singularities,
then $A$ has rational singularities.
\end{Cor}

Next, we can employ the methods in the author's previous work \cite{Splinters} and part of \cite{DTopen} on splinters to the study of birational derived splinters.
As in \cite{Splinters}, we use the notion of a G-ring to state our results in maximal generality.
Unfamiliar readers can refer to \citestacks{07GG} and \cite[(34.A)]{Matsumura}.

The following main results are parallel to \cite[Theorem 1.3]{Splinters}, \cite[Theorem 2.14]{Splinters}, \cite[Theorem 1.0.1]{DTopen}, and
\cite[Theorem 1.1]{Splinters}.
Note that for the results in characteritic $p$,
we involve $F$-purity, since it is required by our method.
See \S\ref{sec:Q} for questions about possible strengthening of these results.
%We hope that the results hold without $F$-purity, or even without a restriction on the chacteristic.

\begin{Thm}[=Theorem \ref{thm:ScalarFoverRMIRROR}]\label{thm:ScalarFoverR}
Let $(S,\fm)\to (S',\fm')$ be a regular homomorphism of Noetherian local rings with $\fm'=\fm S'$.
If $S$ is a birational derived splinter, so is $S'.$
In particular, the completion of a local G-ring that is a birational derived splinter is a birational derived splinter.
\end{Thm}

\begin{Thm}[see Corollary \ref{cor:normallocusoverFpure} and Theorem \ref{thm:openMIRROR}(\ref{FpureOpen})]\label{thm:openoverFpure}
Let $R$ be a Noetherian local $\bF_p$-algebra, and $A$ an $R$-algebra essentially of finite type.
Assume $R$ is $F$-pure, and that $R\to A$ is $F$-pure \cite[(2.1)]{Has10}.

Then the locus of prime ideals $\fp$ of $A$ such that $A_\fp$ is a birational derived splinter (resp., a splinter, a normal local ring) is open in $\Spec(A)$.
\end{Thm}

%Parallel to \cite{DTopen},
%we can prove the following openness result.
\begin{Thm}[=Theorem \ref{thm:openMIRROR}(\ref{Gopen})]\label{thm:open}
Let $A$ be a Noetherian $\bF_p$-algebra.
Assume either that $A$ is $F$-finite, or that $A$ is essentially of finite type over a Noetherian local G-ring.

Then the locus of prime ideals $\fp$ of $A$ such that $A_\fp$ is an $F$-pure birational derived splinter is open in $\Spec(A)$.
\end{Thm}
Here the corresponding result for the normal locus is well-known, see Lemma \ref{lem:Aquasiexcellent} below, and for the splinter locus it is \cite[Theorem 1.0.1]{DTopen}.

%This will allow us to prove the following ascending result.
\begin{Thm}[=Theorem \ref{thm:SmCharpMIRROR}]\label{thm:SmCharp}
%Let $S$ be a plinter of pure characteristic $p$.
Let $S\to A$ be a regular homomorphism of Noetherian $\bF_p$-algebras.
If $S$ is an $F$-pure birational derived splinter, so is $A$. %then $A$ is a splinter.
\end{Thm}

Finally, we mention a curious result, which, as does Theorem \ref{thm:openoverFpure}, fails for a general $F$-pure Noetherian ring $A$, even if we assume $A$ is a G-ring.
See Remark \ref{rem:Openfail}.
%we obtained in the course toward the main theorems.

\begin{Thm}[=Corollary \ref{cor:normalizationofFpure}]
Let $R$ be a Noetherian local $\bF_p$-algebra, $A$ an $R$-algebra essentially of finite type.
Assume $R$ is $F$-pure, and that $R\to A$ is $F$-pure \cite[(2.1)]{Has10}.
Then the normalization of $A$ is finite over $A$.
\end{Thm}

The layout of this paper is as follows.
\S\ref{sec:Prep} contains necessary general results of rings, modules, schemes, and morphisms.
Since our method requires the use of non-Noetherian rings,
we need to introduce and study a variant of proper birational morphisms that is more robust.
This is done in \S\S\ref{subsec:Mmor}.
Another class of morphisms may work as well.
%Besides, a large portion (\S\S\ref{subsec:PerfPsCoh} and \S\S\ref{subsec:limit})
%is dedicated for the limit result, Theorem \ref{thm:NoetherianManifest}(\ref{ManifestLimit}).
In \S\ref{sec:BDS} we study birational derived splinters and a non-Noetherian variant of which defined using the class of morphisms defined in \S\S\ref{subsec:Mmor}.
In \S\ref{sec:lim} we treat direct limit, and in \S\ref{sec:Et} we treat \'etale extensions.
The rest of the main results are proved in \S\ref{sec:ResidueExtn} and \S\ref{sec:OpenandReg}.
We ask some questions in \S\ref{sec:Q}.\\

\textsc{Acknowledgements}. We thank Rankeya Datta, Linquan Ma, Takumi Murayama, and Kevin Tucker, discussions with whom for the previous work \cite{Splinters} persist to be helpful; we also thank János Kollár, Longke Tang, and Chenyang Xu for helpful discussions.

\section{Preparations}\label{sec:Prep}

For a ring $A$, the expression ``$\dim A=0$'' means ``every prime ideal of $A$ is maximal.''
The reduction of $A$ is denoted by $A_{\red}$.
We do not distinguish between $A$-modules and quasi-coherent $\cO_{\Spec(A)}$-modules, and similarly $D(A)=D(QCoh(\cO_{\Spec(A)}))=D_{QCoh}(\cO_{\Spec(A)})$, cf. \citestacks{06Z0}.

\subsection{Residue extensions}

\begin{Def}\label{def:ResidueExtension}
A \emph{residue extension} is a flat local map $\varphi:(A,\fm,k)\to (B,\fn,l)$ of Noetherian local rings such that $\fn=\fm B$.
A \emph{separable residue extension} is a residue extension where $l/k$ is separable \citestacks{030O};
a \emph{regular residue extension} is a residue extension where $\varphi$ is regular \citestacks{07BZ}.
\end{Def}

A regular residue extension is always separable (cf. \citestacks{0322}).
The converse holds in the case $A$ is a G-ring:

\begin{Lem}[\cite{Andre}]
\label{lem:Gseparable=regular}
Let $\varphi:A\to B$ be a separable residue extension of Noetherian local rings.
If $A$ is a G-ring then $\varphi$ is a regular residue extension.
\end{Lem}
\begin{proof}
We have a commutative diagram
\[\begin{CD}
A @>{\varphi}>> B\\
@V{\iota_A}VV @V{\iota_B}VV\\
A^\wedge @>{\varphi^\wedge}>> B^\wedge
\end{CD}\]
of flat local maps of Noetherian local rings.
Since $\varphi$ is a separable residue extension, so is $\varphi^\wedge$,
thus $\varphi^\wedge$ is regular \citestacks{07PM},
and so is $\iota_A$ by the definition of a G-ring.
Since $\iota_B$ is faithfully flat we see $\varphi$ regular by \citetwostacks{07QI}{07NT}.
\end{proof}

\subsection{Total fraction rings}

\begin{Lem}\label{lem:Partial=TotalFraction}
Let $A$ be a ring. %, $K$ its total fraction ring.
Let $S$ be a multiplicative subset of $A$ that consists of nonzerodivisors.
If $\dim(S^{-1}A)=0$, then the total fraction ring of $A$ is $S^{-1}A$.
\end{Lem}
\begin{proof}
It suffices to show that every nonzerodivisor  $x\in S^{-1}A$ is invertible.
If not, take a maximal ideal $\fm$ of $S^{-1}A$ that contains $x$.
Since every prime ideal of $S^{-1}A$ is maximal, $\fm$ is a minimal prime and $x$ is nilpotent in $(S^{-1}A)_{\fm}$, contradiction.
\end{proof}

\begin{Lem}\label{lem:TotalFractionProduct}
Let $X$ be a set, $A_x\ (x\in X)$ be a family of reduced rings with total fraction rings $K_x$.
If every $A_x$ has finitely many minimal primes, then $K:=\prod_x K_x$ is the total fraction ring of $A:=\prod_x A_x$, and $\dim K=0$.
\end{Lem}
\begin{proof}
Let $S_x$ be the set of nonzerodivisors of $A_x$.
Then the set of nonzerodivisors of $A$ is $S=\prod_x S_x$, and $S^{-1}A=K$, so $K$ is the total fraction ring of $A$.
By \citetwostacks{02LX}{00EW}, $K_x$ is a product of fields, hence so is $K$.
Thus $\dim K=0$, see \citetwostacks{092G}{092F}.
%It is elementary to see that $K$ is the   Lemma \ref{lem:Partial=TotalFraction}.
\end{proof}

\subsection{M-morphisms}\label{subsec:Mmor}
To define birational derived splinters, we need to make sense of what is ``birational.''
We use the following class of morphisms that behaves better with respect to limit and base change.
The name ``M-morphism'' indicates its close relation to modifications \citestacks{0AAZ}.
See Lemma \ref{lem:NoetherianBirationalandM} below.

\begin{Def}\label{def:DEFM}
Let $A$ be a ring.
Let $f:X\to S=\Spec(A)$ be a morphism of schemes.
We say $f$ is an \emph{M-morphism}
if $f$ is proper and of finite presentation, and for the total fraction ring $K$ of $A_{\red}$, the base change of $f$ to $\Spec(K)$ is an isomorphism.
\end{Def}

Note that we do not assume the irreducible components of $X$ lying above those of $S$.
%We use the very general notion of a dominant morphism, \citestacks{01RJ}.
%
\begin{Lem}\label{lem:Mbasechange}
Let $A\to B$ be a ring map and %$K$ the total fraction ring of $A_{\red}$,  
$f:X\to \Spec(A)$ an M-morphism.
%Write $T=\Spec(B)$ and $f_T:X_T\to T$ the base change of $f$.

Assume that every nonzerodivisor of $A_{\red}$ is mapped to a nonzerodivisor of $B_{\red}$.
Then the base change of $f$ to $B$ is an M-morphism.
%Then $f_B:X\times_S\Spec(B)\to \Spec(B)$ is an M-morphism
%
%Assume further that $T\to S$ is dominant and $\dim K=0$.
%Then if $f_T$ is an M-morphism, so is $f$.
\end{Lem}
\begin{proof}
We may assume both $A$ and $B$ are reduced.
Now every nonzerodivisor of $A$ is mapped to a nonzerodivisor of $B$,
so the total fraction ring $L$ of $B$ is naturally a $K$-algebra where $K$ is the total fraction ring of $A$.
This proves the lemma.
%This yields the first part of the lemma.
%
%If $T\to S$ is dominant, then $A\to B$ is injective, see \citetwostacks{056A}{056B}.
%Thus $K\to L$ is also injective, hence $\Spec(L)\to\Spec(K)$ is dominant.
%Since $\dim K=0$, $\Spec(L)\to\Spec(K)$ is surjective, see \citestacks{02JQ}.
%Since every $K$-module is flat (\citestacks{092F}), $L$ is faithfully flat over $K$, which gives the second part of the lemma.
\end{proof}

\begin{Lem}\label{lem:LocalizeProperFP}
Let $A$ be a ring, $U$ a multiplicative subset of $A$, $B=U^{-1}A$. %$S=\Spec(A),T=\Spec(B)$.
%Let $K$ be the total fraction ring of $A$ and assume $\dim K=0$; also assume $A$ is reduced and integrally closed in $K$.

Let $S=\Spec(A)$ and let $T$ be either a quasi-compact open subscheme of $S$ or $\Spec(B)$.
Let $g:Y\to T$ be a proper morphism of finite presentation.
Then there exists a proper morphism $f:X\to S$ of finite presentation such that $X\times_S T\cong Y$.
\end{Lem}
\begin{proof}
%First, we show that there exists a proper morphism of finite presentation $X_0\to S$ such that $X_0\times_S T=Y$.
Assume first that $T=\Spec(B)$.
We know that $B=\colim_{u\in U}A_u$,
so there exists an element $u\in U$ and a proper morphism $f_1:X_1\to \Spec(A_u)$ of finite presentation such that $X_1\times_{\Spec(A_u)} T\cong Y$, see \citetwostacks{01ZM}{081F}.
So we may replace $T$ by $\Spec(A_u)$, and we may always assume $T$ a quasi-compact open subscheme of $S$.
By Nagata compactification \citestacks{0F41}, there exists an open immersion $j:Y\to X_2$ where $X_2$ is proper over $S$.
Note that $j$ is quasi-compact as $T$ is.
We may replace $X_2$ by the scheme-theoretic image of $j$ to assume $j$ scheme-theoretically dominant, in particular
$X_2\times_S T=Y$.
Note that $X_2$ is not necessarily of finite presentation over $S$.
However, by \citestacks{09ZR},
$X_2=\lim X_i$ where the transition maps are closed immersions and each $X_i$ is proper of finite presentation over $S$.
The system $Y=X_2\times_S T\to X_i\times_S {T}$ of morphisms of $X_i$-schemes satisfies \citestacks{081D} (as $X_2\to X_i$ is a closed immersion), so by \citestacks{081E},
we see that there exists an $i_0$ such that $Y= X_{i_0}\times_S {T}$.
We take $X=X_{i_0}$.
\end{proof}

\begin{Lem}\label{lem:LocalizeM}
Let $A$ be a ring, $U$ a multiplicative subset of $A$, $B=U^{-1}A$. %$S=\Spec(A),T=\Spec(B)$.
Let $K$ be the total fraction ring of $A$ and assume $\dim K=0$; also assume $A$ is reduced and integrally closed in $K$.

Let $S=\Spec(A)$, and $T=\Spec(B)$. % be either an affine open subscheme of $S$ or $\Spec(B)$.
Let $g:Y\to T$ be an M-morphism.
Then there exists an M-morphism $f:X\to S$ such that $X\times_S T\cong Y$.
\end{Lem}
\begin{proof}
By Lemma \ref{lem:LocalizeProperFP}, there exists an $X_{0}$ proper of finite presentation over $S$ such that $X_{0}\times_S T=Y$.
In particular $h:X_{0}\times_S \Spec(K)\to \Spec(K)$ is proper of finite presentation. % over $\Spec(K)$.
We note that $h$ is flat since $K$ is reduced and $\dim K=0$, see \citestacks{092F}.
By \citetwostacks{0D4J}{02LS}, there exists a unique clopen subset $E_1$ of $\Spec(K)$ such that $h$ is finite over $E_1$ and that $h$ does not have any finite fiber outside $E_1$.
Then $h_*\cO_{X_{0}\times_S \Spec(K)}$ is finite locally free over $E_1$.
We now see that there is a unique clopen subset $E_0$ of $E_1$ such that $h$ is an isomorphism over $E_0$ and the fibers of $h$ outside $E_0$ are either empty or non-trivial.
Since $A$ is integrally closed in $K$, $E_0$ comes from a clopen subset $E$ of $S$, and it is now clear that $X:=\left(X_{0}\times_S E\right)\sqcup \left(S\setminus E\right)\to S$ is a well-defined M-morphism.

Finally, we show $X\times_S T\cong Y$.
Since $X_{0}\times_S T\cong Y$, it suffices to show that the image of the morphism $T\to S$ is contained in $E$.
As $\dim K=0$, $\dim (U^{-1} K)=0$, so $U^{-1} K$ is the total fraction ring of $B$, see Lemma \ref{lem:Partial=TotalFraction}.
Therefore by construction and the fact that $g$ is an M-morphism, the image of $\Spec(U^{-1}K)$ in $\Spec(K)$ is contained in $E_0$.
Since $\Spec(U^{-1} K)\to T$ is dominant (cf. \citestacks{00EW}) the image of $T$ is therefore contained in $E$ as desired.
\end{proof}

\begin{Lem}\label{lem:OneBmodif}
Let $A$ be an integral domain and $f:X\to \Spec(A)$ a proper surjective morphism of finite presentation.
If there exists a (nonzero) $A$-algebra $B$ such that $f_B:X\times_{\Spec(A)}\Spec(B)\to \Spec(B)$ is an M-morphism, then $f$ is an M-morphism.
\end{Lem}
\begin{proof}
Replace $B$ by its reduction, and then a residue field of its total fraction ring, we may assume $B$ a field.
Let $\fp=\ker(A\to B)$.
Since $A$ is an integral domain it suffices to show $f_{A_\fp}$ an M-morphism.
Thus we may assume $(A,\fp,k)$ local.
$B$ is now a faithfully flat $k$-algebra, so $f_k$ is an isomorphism.
By Zariski's Main Theorem (cf. \citestacks{02UP}) $f$ is finite.
By Nakayama's Lemma $f$ is a closed immersion.
Since $f$ is surjective and $A$ is reduced, $f$ is an isomophism, as desired.
\end{proof}

\begin{Cor}\label{cor:Mlimitdomain}
Let $(A_i)_i$ be a direct system of integral domains
%rings with injective transition maps,
and let $A=\colim_i A_i$.
Let $f:X\to S:=\Spec(A)$ be an M-morphism.
Then there exists an index $i$ and an M-morphism $f_i:X_i\to S_i:=\Spec(A_i)$ whose base change to $A$ is $f$.
\end{Cor}
\begin{proof}
There exists an index $0$ and a morphism $f_0:X_0\to S_0$ of finite presentation whose base change to $A$ is $f$, see \citestacks{01ZM}.
We may assume $f_0$ proper surjetive by \citetwostacks{081F}{07RR}.
Then $f_0$ is an M-morphism by Lemma \ref{lem:OneBmodif}.
\end{proof}

%\begin{Lem}\label{lem:MlimitNoetherian}
%Let $A=\colim_i A_i$ be a direct limit of rings.
%Assume that each $A_i$ is a finite product of integral domains.
%Let $f:X\to S:=\Spec(A)$ be an M-morphism.
%There exists an index $i$ and an M-morphism $f_i:X_i\to S_i:=\Spec(A_i)$ whose base change to $A$ is $f$.
%\end{Lem}
%\begin{proof}
%\ref{lem:Mlimit},
%there exists an index $0$ and a proper morphism $f_0:X_0\to S_0$ of finite presentation whose base change to $A$ is $f$.
%We may assume $f_0$ surjective by \citestacks{07RR}.
%Over each irreducible component of $S_0$ that intersects the image of $S$,
%$f_0$ is an M-morphism by Lemma \ref{lem:OneBmodif} (and Lemma \ref{lem:Mbasechange}).
%Over other irreducible components we can just replace $f_0$ by the identity map, as we did in the proof of Lemma \ref{lem:LocalizeM}.
%\end{proof}

Birational morphisms are defined as in \citestacks{01RN}.
\begin{Lem}\label{lem:NoetherianBirationalandM}
Let $A$ be a Noetherian reduced ring, $f:X\to S=\Spec(A)$ be an M-morphism.
Then there exists a reduced closed subschme $Z$ of $X$ such that $Z\to S$ is birational.

Conversely, a proper birational morphism $g:Y\to S$ is an M-morphism.
\end{Lem}
\begin{proof}
For the first part we can just take $Z$ to be the scheme-theoretic image of $X\times_S\Spec(K)$ in $X$ where $K$ is the total fraction ring of $A$.

Every proper $A$-scheme is of finite presentation since $A$ is Noetherian.
Therefore the second part follows immediately from \citestacks{0BAB}.
\end{proof}

\subsection{Splitting in the derived category: trace ideals}\label{subsec:Derived}
Let $A$ be a ring, $f:X\to S=\Spec(A)$ be a qcqs morphism of schemes.
Then $Rf_*\cO_X$ lies in $D^b(A)$, see \citestacks{08D5}.
The formation of $Rf_*\cO_X$ commutes with flat base change, see \citestacks{08IB}.
Therefore if $A\to Rf_*\cO_X$ splits, then the same is true after any flat base change $A\to B$.

The canonical map $A\to Rf_*\cO_X$ splits if and only if the induced map  $\Hom_{D(A)}(Rf_*\cO_X,A)\to \Hom_{D(A)}(A,A)$ of $A$-modules is surjective, by general category theory.
The formation of this map commutes with flat base change if $Rf_*\cO_X$ is pseudo-coherent \citestacks{08CB}, see \citetwostacks{0A6A}{0A64}.
This holds when $A$ is Noetherian and $f$ is proper (cf. \citestacks{08E8}).

Using the identification $A=\Hom_{D(A)}(A,A)$, the morphism $f$ determines an ideal of $A$, namely the image of $\Hom_{D(A)}(Rf_*\cO_X,A)\to \Hom_{D(A)}(A,A)=A$.
We denote this ideal by $\ft(f)$ or $\ft(X/A)$.
As noted in the previous paragraph, $A\to Rf_*\cO_X$ splits if and only if $\ft(f)=A$,
and the formation of $\ft(f)$ commutes with flat base change if $A$ is Noetherian and $f$ is proper.
Note that in general, for any flat ring map $A\to B$, we have at least $\ft(X/A)B\subseteq \ft(X_B/B)$. % as noted in the first paragraph.

\section{Birational derived splinters}\label{sec:BDS}

In this section we study some fundamental properties of birational derived splinters, and the non-Noetherian variant MDSs as defined below.
Many results are parallel to those in \cite[\S 3]{Splinters}, and we do refrain from writing too many ``cf.''s.
In this section and what follows,
we shall use the term ``birational derived splinter'' for Noetherian rings and ``MDS'' for general rings to state our results, and these notions are the same for Noetherian rings, Lemma \ref{lem:characterizeNoetherianMDS}.
%We also obtained a satisfactory limit result in the Noetherian case, Theorem \ref{thm:MDSlimitHARD}, overcoming the difficulty discussed in \cite[Remark 5.3.3]{ADVal} in this case.

\begin{Def}\label{def:MDS}
A ring $A$ is an \emph{M-derived splinter}, shorthand \emph{MDS}, if %$A$ is reduced, and
for every M-morphism (Definition \ref{def:DEFM}) $f:X\to\Spec(A)$ the map $A\to Rf_*\cO_X$ splits in $D(A)$.
\end{Def}

\begin{Lem}\label{lem:MDSnormal}
An MDS $A$ is reduced and integrally closed in its total fraction ring.
\end{Lem}
\begin{proof}
Let $a\in A$ be nilpotent.
Then $\Spec(A/aA)\to \Spec(A)$ is an M-morphism (Definition \ref{def:DEFM}), so $A\to A/aA$ is a split map of $A$-modules.
Thus $a=0$ and $A$ is reduced.

Let $K$ be the total fraction ring of $A$ and $B\subseteq K$ be a finite $A$-algebra.
We have $B=\colim_i B_i$ where $B_i$ are finite of finite presentation over $A$ and the transition maps $B_i\to B_j$ are surjective, see \citestacks{09YY}.
Using either \citestacks{081E} or \citestacks{00QO} we see that $(B_i)_K=K$ for large $i$, so $\Spec(B_i)\to\Spec(A)$ is an M-morphism.
Thus $A\to B_i$ is split, and we see $aB_i\cap A=aA$.
Taking colimit, we see that $aB\cap A=aA$ for all $a\in A$.
Since $B\subseteq K,A=B$, as desired.
\end{proof}

\begin{Def}
A Noethrian ring $A$ is a \emph{birational derived splinter}, %shorthand \emph{BDS},
if $A$ is reduced and for every proper birational morphism $f:X\to\Spec(A)$ the map $A\to Rf_*\cO_X$ splits in $D(A)$.
%
%A Noetherian ring $A$ is a \emph{birational derived splinter}, %shorthand \emph{BDS},
%if $A$ is a direct product of Noetherian integral birational derived splinters.
\end{Def}

Here, the notion of a birational morphism to a not necessarily integral scheme is defined in \citestacks{01RN}.

This is the same notion as Definition \ref{def:MDS} for Noetherian rings:

\begin{Lem}\label{lem:characterizeNoetherianMDS}
Let $A$ be a Noetherian ring.
Then $A$ is an MDS if and only if $A$ is a birational derived splinter, in which case $A$ is normal.
\end{Lem}

\begin{proof}
A Noetherian MDS is normal by Lemma \ref{lem:MDSnormal}, so we always have $A$ reduced.
The equivalence of MDS and birational derived splinter follows immediately from Lemma \ref{lem:NoetherianBirationalandM}.
\end{proof}

In equal characteristic zero, we have the following characterizaition of birational derived splinters.

\begin{Prop}[cf. {\cite{KovacsOther}}]\label{prop:charactrizeQBDS}
Let $A$ be a quasi-excellent $\bQ$-algebra.
Then the followings are equivalent.
\begin{enumerate}
    \item\label{QringBDS} $A$ is a birational derived splinter.
    
    \item\label{someResolRatlsplit} There exists a resolution of singularities $f:X\to \Spec(A)$ such that $A\to Rf_*\cO_X$ splits.
    
%    \item\label{someResolRatl} There exists a resolution of singularities $f:X\to \Spec(A)$ with $A=Rf_*\cO_X$.
    
    \item\label{allResolRatl} For any resolution of singularities $f:X\to \Spec(A)$, we have $A=Rf_*\cO_X$.
\end{enumerate}

If the equivalent conditions above hold, we say $A$ \emph{has rational singularities}, cf. \cite[Definition 2.76]{Kol13}.
%$ if and only if $A$ has rational singularities.
\end{Prop}

%Here, we call a quasi-excellent $\bQ$-algebra $A$ to have rational singularities if $A$ is normal and for all resolution of singularities $f:X\to \Spec(A)$, one has $R=Rf_*\cO_X$.
%Such resolutions exist by \cite{Tem12}.
\begin{proof}
In all cases $A$ is normal.
The existence of resolutions \cite[Theorem 1.2.1]{Tem12} 
shows that (\ref{allResolRatl}) implies (\ref{QringBDS}) and that (\ref{QringBDS}) implies (\ref{someResolRatlsplit}).
%Trivially (\ref{someResolRatl}) implies (\ref{someResolRatlsplit}),
It suffices to show (\ref{someResolRatlsplit}) implies (\ref{allResolRatl}).
%We shall show below that either one of the latter implies the former.
%We shall show (\ref{someResolRatl})
%as well as .

We make the following observation.
Let $\fp\in\Spec(A)$ and let $A'$ be the completion of $A_\fp$.
Then $A_\fp\to A'$ is regular by definition,
so $A'$ is normal
and if $f:X\to\Spec(A)$ is a resolution, then the base change $f':X':=X\times_{\Spec(A)}\Spec(A')\to \Spec(A')$
is also a resolution, see for example \citetwostacks{0C22}{033A}.
Since $\fp$ was arbitrary,
by flat base change (cf. \S\S\ref{subsec:Derived}) and descent,
we may always assume $A$ complete local.
In particular, $A$ has a dualizing complex \citestacks{0BFR} and is excellent.

Assume (\ref{someResolRatlsplit}).
\cite{CR15} and the existence of resolutions \cite[Theorem 1.2.1]{Tem12} show that for all resolutions $g:Y\to \Spec(A)$,
$A\to Rg_*\cO_Y$ splits.
Since we have the Grauert-Riemenschneider vanishing theorem \cite[Theorem A]{Takumi} in this case,
\cite[Corollary 2.75]{Kol13}
shows that $Rg_*\cO_Y=A$, as desired.
%(\ref{QringBDS}) implies (\ref{allResolRatl}).
%since we have the Grauert-Riemenschneider vanishing theorem in this case ($L=0$ in \cite[Theorem A]{Takumi}).
%Thus $A'=Rf'_*\cO_{X'}$, see \cite[Corollary 2.75]{Kol13}.
%Since a resolution of $A$ becomes a resolution of $A'$ after base change, and
%
\end{proof}

\begin{Rem}
Note that a quasi-excellent $\bQ$-algebra that has rational singularities is Cohen-Macaulay,
see \cite[Proposition 2.77]{Kol13}.
Thus by Lemma \ref{lem:LocalizeMDS} below and \cite[(34.A)]{Matsumura}
we see a G-ring that contains $\bQ$ and is a birational derived splinter is Cohen-Macaulay.

If we drop the excellence assumption completely, then the definition of rational singularities above may be vacuous, cf. \cite[Proposition 7.9.5]{EGA4_2}.
The author does not know, for example, if all Noetherian birational derived splinter containing $\bQ$ are Cohen-Macaulay.
\end{Rem}

\begin{Lem}\label{lem:ProdMDS}
Let $X$ be a set and $A_x\ (x\in X)$ a family of MDSs.
Then $A:=\prod_{x\in X}A_x$ is an MDS.
\end{Lem}
\begin{proof}
We note first that a product of reduced rings is reduced, so by Lemma \ref{lem:MDSnormal} we see $A$ and all $A_x$ reduced.

Let $g:Y\to \Spec(A)$ be an M-morphism and $g_x:Y_x\to \Spec(A_x)$ be the base change of $g$ along the projection map.
Note that a nonzerodivisor of $A$ is just an element all of whose coordinates are nonzerodivisors.
Thus by Lemma \ref{lem:Mbasechange}, $g_x$ is an M-morphism and thus $A_x\to Rg_{x*}\cO_{Y_x}$ splits in $D(A_x)$.
Taking product, we see that
$A\to \prod_x Rg_{x*}\cO_{Y_x}$ splits in $D(A)$.
Since $A\to A_x$ is flat, $Rg_{x*}\cO_{Y_x}=Rg_*\cO_Y\otimes^L_A A_x$,
so $A\to \prod_x Rg_{x*}\cO_{Y_x}$  factors through $A\to Rg_*\cO_Y$ and thus the latter map splits.
\end{proof}
%\noindent\emph{Remark}.

\begin{Rem}
One can do the proof in terms of explicit complexes if not comfortable with infinite products in the derived category.
Cover $Y$ by affine opens $U^{(j)}\ (1\leq j\leq n)$.
Then we have the \v{C}ech complex
$\check{C}^\bullet=\check{C}^\bullet(\cU,\cO)$ that computes $Rg_*\cO_Y$,
and we see that the complex $\check{C}^\bullet_x:=\check{C}^\bullet\otimes_A A_x$ is a \v{C}ech complex that computes $Rg_{x*}\cO_{Y_x}$.
Denote by $\varphi_x$ the canonical map $A_x\to \check{C}^\bullet_x$,
so $\varphi_x$ splits in $D(A_x)$,
which means that there exists an $A_x$-complex $F_x^\bullet$ together with a map $\psi_x:F_x^\bullet\to A_x$ and a quasi-isomorphism $s_x:F_x^\bullet\to \check{C}^\bullet_x$
such that $H^0(\psi_x)\circ H^0(s_x)^{-1}\circ H^0(\varphi_x)=\mathrm{id}_{A_x}$.
%and that there exists a cohomology class $\alpha_x\in H^0(F_x^\bullet)$
%such that $H^0(\varphi_x\circ\psi_x)(\alpha_x)=H^0(\varphi_x)(1)\in H^0(\check{C}^\bullet\otimes_A A_x)$.
Taking product of these data we get an $A$-complex $F^\bullet=\prod_x F_x^\bullet$, a map $\psi=\prod_x \psi_x$ and a quasi-isomorphism $s=\prod_x s_x$
%a cohomology class $\alpha=(\alpha_x)_{x\in X}$
satisfying the same conditions and we see
$\prod_x A_x\to\prod_x\check{C}^\bullet_x$
splits in $D(A)$.
Finally, we have the canonical factorization
$A=\prod_x A_x\to \check{C}^\bullet\to\prod_x\check{C}^\bullet_x$
in the category of cochain complexes of $A$-modules.
Thus $A\to \check{C}^\bullet$ splits in $D(A)$ as desired.
\end{Rem}

\begin{Lem}\label{lem:LocalizeMDS}
Let $A$ be an MDS with total fraction ring $K$.
Assume that $\dim K=0$.
Then any localization $S^{-1}A$ is an MDS.
\end{Lem}
\begin{proof}
%Reducedness localizes, so w
We need to find a splitting $S^{-1}A\to Rg_*\cO_Y$ for every M-morphism $g:Y\to\Spec(S^{-1}A)$.
By Lemma \ref{lem:MDSnormal} $A$ is reduced and integrally closed in $K$, so Lemma \ref{lem:LocalizeM} applies and shows that $g$ extends to an M-morphism $f:X\to \Spec(A)$, and $A\to Rf_*\cO_X$ splits.
By flat base change, see \S\S\ref{subsec:Derived}, $S^{-1}A\to Rg_*\cO_Y$ splits as desired.
%$S^{-1}A$ is integrally closed in $S^{-1}K$ \citestacks{0307}.
%Since $\dim K=0$, $\dim(S^{-1}K)=0$, so $S^{-1}K$ is the total fraction ring of $S^{-1}A$
%Immediate from ,
%which we can apply by assumptions and Lemma \ref{lem:MDSnormal}.
\end{proof}

Recall that a ring map $A\to B$ is \emph{pure} if it is universally injective as a map of $A$-modules.
See \citestacks{058H}.
The fact that $B$ is not necessarily assumed Noetherian in the following result is essential later on.
%For a Noetherian version see Lemma \ref{cor:NoetherianPureDescendMDS}.

\begin{Prop}\label{prop:PureDescendMDS}
Let %$A$ be a Noetherian ring,
$A\to B$ be a pure ring map
that sends nonzerodivisors to nonzerodivisors (for example, a faithfully flat ring map).
Assume that $A$ is Noetherian and that $B$ is an MDS.
Then $A$ is a birational derived splinter.
\end{Prop}
\begin{proof}
Since $B$ is an MDS, it is reduced by Lemma \ref{lem:MDSnormal}, so $A$ is reduced as a pure ring map is injective.
%We need to show that $A$ is an MDS,
%see Lemma \ref{lem:characterizeNoetherianMDS}.

Let $f:X\to S=\Spec(A)$ be an M-morphism,
and $T=\Spec(B)$.
The base change $f_T:X_T\to T$ is also an M-morphism, see Lemma \ref{lem:Mbasechange}.
%, which we can apply by assumptions and by the fact $A$ is reduced.
Since $B$ is an MDS, $B\to Rf_{T*}\cO_{X_T}$ splits in $D(B)$.

Let $\fm$ be a maximal ideal of $A$ and let $A'$ be the completion of $A$ at $\fm$.
%Since the formation of $Rf_{T*}\cO_{X_T}$ commutes with flat base change,
By flat base change, see \S\S\ref{subsec:Derived},
$B'\to Rf'_{T*}\cO_{X'_T}$ splits in $D(B')$,
where $(-)'$ is the base change along $A\to A'$.
We also note that $A'\to B'$ is pure, and that $A'$ is a Noetherian complete local ring, so $A'\to B'$ splits as a map of $A'$-modules by \cite[Lemma 1.2]{Fed83}.
Therefore $A'\to Rf'_{T*}\cO_{X'_T}$ splits in $D(A')$, and since this map factors through $A'\to Rf'_{*}\cO_{X'}$,
we see that $A'\to Rf'_{*}\cO_{X'}$ splits in $D(A')$ as well.

Using the observations made in \S\S\ref{subsec:Derived}, we see that the ideal $\ft(f)$ satisfies $\ft(f)A'=A'$.
Since $\fm$ was arbitrary, we see $\ft(f)=A$,
as desired.
\end{proof}

\begin{Cor}\label{cor:MDSisLocal}
Let $A$ be a Noetherian ring.
The followings are equivalent.

(1) $A$ is a birational derived splinter.

(2) $A_\fp$ is a birational derived splinter for all prime ideals $\fp$ of $A$.

(3) $A_\fm$ is a birational derived splinter for all maximal ideals $\fm$ of $A$.
\end{Cor}
\begin{proof}
We have (1)$\Longrightarrow$(2) by Lemma \ref{lem:LocalizeMDS}, and (2)$\Longrightarrow$(3) is trivial.

The implication (3)$\Longrightarrow$(1) follows from a similar but much simpler argument as in the proof of Proposition \ref{prop:PureDescendMDS}, but we can also obtain the result directly as follows.
For a Noetherian ring $A$, the map $A\to \prod_{\fm} A_{\fm}$, where $\fm$ runs over all maximal ideals of $A$, is faithfully flat, cf. \citetwostacks{05CZ}{05CY}.
Our result now follows from Lemma \ref{lem:ProdMDS} and Proposition \ref{prop:PureDescendMDS}.
\end{proof}

In the Noetherian case,
we also have the following strengthening of Proposition \ref{prop:PureDescendMDS}.
Together with Proposition \ref{prop:charactrizeQBDS} and \cite[Lemma 2.2]{HH95},
we recover \cite[Theorem C]{Takumi} (note that our definition is slightly different from \cite[Definition 7.2]{Takumi}).
\begin{Cor}\label{cor:NoetherianPureDescendMDS}
Let $A\to B$ be a cyclically pure ring map, that is, for all ideals $I\subseteq A$ we have $IB\cap A=I$.
If $B$ is a Noetherian birational derived splinter, so is $A$,
and $A\to B$ is pure.
\end{Cor}
\begin{proof}
From Lemma \ref{lem:characterizeNoetherianMDS} we see $B$ normal.
By \cite[Corollary 3.12]{Has10} (which ultimately comes from \cite{Hoc77}), $A$ is a Noetherian normal ring and $A\to B$ is pure.

For every $\fp\in\Spec(A)$, $A_\fp\to B_\fp$ is pure,
so \cite[Lemma 2.2]{HH95} shows that
there exists $\fq\in \Spec(B)$ such that $\fq\cap A\subseteq \fp$ and
$A_\fp\to B_\fq$ is pure.
Now $A_\fp$ and $B_\fq$ are integral domains, so $A_\fp\to B_\fq$
sends nonzerodivisors to nonzerodivisors.
By Proposition \ref{prop:PureDescendMDS} (and Lemma \ref{lem:LocalizeMDS}), $A_\fp$ is a birational derived splinter.
Since this holds for all $\fp\in\Spec(A)$ we see $A$ is a birational derived splinter from Corollary \ref{cor:MDSisLocal}.
\end{proof}

%%The limit result above has a quite restrictive flatness assumption. % that the transition maps are flat.
%This is indeed the case, for example, for the
%\'etale algebras occuring in the Henselization of a local domain (cf. the proof of Corollary \ref{cor:IndEtMDS}).
%Below we establish another, rather artificial, limit result, for the purpose of reducing Theorem \ref{thm:SmCharp} to the case $A$ is smooth over $S$.
%However, in the Noetherian case we have the following satisfying limit result:
%
%Finally, let us prove our result about direct limit of birational derived splinters.
%The proof is short, but the reader should be aware that hard work has already been done in \S\S\ref{subsec:PerfPsCoh} and \ref{subsec:limit}.
%
\section{Direct limit}\label{sec:lim}

In this section we consider the direct limit of a system of birational derived splinters.
We need to do some work on derived categories and direct limit.

\textsc{Convention}:
by a complex of modules over a ring $A$ we mean a cochain complex in the abelian category of $A$-modules.
We usually denote it by $\square^\bullet$ where $\square$ is a capital Latin letter,
in which case $\square$ denotes the corresponding object in the derived category $D(A)$.

\subsection{Splitting in the derived category: from perfect to pseudo-coherent}\label{subsec:PerfPsCoh}

In this subsection we establish the following result.
This is inspired by a notion of purity in general categories, cf. \cite[\S 2.D]{AR}.

\begin{Lem}\label{lem:PerfToPseudoCoh}
Let $R$ be a ring, $E$ a bounded object in $D(R)$, $F$ a pseudo-coherent object in $D(R)$ \citestacks{064Q}.

Let $\varphi:E\to F$ be a map in $D(R)$.
If for all perfect objects \citestacks{0657} $K\in D(R)$ and all factorizations $E\to K\to F$ of $\varphi$ the map $E\to K$ is a split mono in $D(R)$,
then $\varphi$ is a split mono in $D(R)$.
\end{Lem}
\begin{proof}
By definition, we may represent the object $F$ by a complex $F^\bullet$ of finite free $R$-modules such that $F^i=0$ for $i$ outside an interval $(-\infty,b]\ (b\in \bZ)$.
Choose any complex $E^\bullet$ that represents $E$ together with a map of complexes $\varphi^\bullet:E^\bullet\to F^\bullet$ that represent $\varphi$.
Since $E$ is bounded, we may assume that $E^i$ are zero for $i$ outside an interval $[a,b']\ (a,b'\in \bZ,a\leq b')$.
%;
%see the proof of \citestacks{0658}.
%Now the projectivity of each $E^i$ implies that $\varphi$ is represented by a map $\varphi^\bullet:E^\bullet\to F^\bullet$ of complexes.

Next, let $K^\bullet$ be the complex defined by $K^i=F^i\ (i\geq a-1)$ and $K^i=0\ (i<a-1)$.
(This is the stupid truncation $\sigma_{\geq a-1}F^\bullet$, see \citestacks{0118}.)
Then $K^\bullet$ is a perfect complex of $\operatorname{Tor}$-amplitude $[a-1,b]$ that canonically maps to $F^\bullet$,
and the map $\varphi^\bullet$ factors as
$E^\bullet\xrightarrow{\psi^\bullet}K^\bullet\to F^\bullet$.

Now, by assumption, the map $\psi:E\to K$ is a split mono in $D(R)$.
This means that there exists a quasi-isomorphism of complexes $s^\bullet:K'^\bullet\to K^\bullet$
and a map of complexes $\psi'^\bullet:K'^\bullet\to E^\bullet$
such that $\psi'\circ s^{-1}\circ \psi=\mathrm{id}_E$ in $D(R)$.
If $s'^\bullet:K''^\bullet\to K'^\bullet$ is a quasi-isomorphism of complexes, then we may replace $K',s,\psi'$ by $K'',s\circ s',\psi'\circ s'$.
Therefore %again %by the proof of \citestacks{0658}
we may assume $K'^i$ are zero for $i\notin [a-1,b']$.

Now, the composition $
\theta:K'^\bullet\xrightarrow{s}K^\bullet\to F^\bullet$ is bijective on $H^i\ (i\geq a)$ and surjective on $H^{a-1}$.
Thus one can find a complex $F'^\bullet$ with $F'^i=K^i$ when $i\geq a-1$ and a quasi-isomorphism $s'^\bullet:F'^\bullet\to F^\bullet$ that agrees with $\theta$ on degrees $\geq a-1$.
This is by the same process as that to find a projective resolution of a complex, see for example the proof of \citestacks{05T7}.
Since $E^i=0$ for $i<a$ and $F'^i=K'^i$ for $i\geq a-1$, $\psi'^\bullet:K'^\bullet\to E^\bullet$ extends canonically to a map $\psi''^\bullet:F'^\bullet\to E^\bullet$ of complexes by putting $\psi''^{i}=0$ when $i<a-1$.
Since $F$ and $K$ also only differ in degrees $<a-1$, $\psi''\circ s'^{-1}\circ\varphi=\psi'\circ s^{-1}\circ \psi=\mathrm{id}_E$, as desired.
\end{proof}

\subsection{Splitting in the derived category: limits}\label{subsec:limit}

In this subsection fix a direct system $(R_i)_{i\in\Lambda}$ of rings with colimit $R$.

Let $C_i^\bullet$ be complexes of $R_i$-modules equipped with compatible maps of $R_i$-complexes $C_i^\bullet\to C_j^\bullet$ for $i\leq j$.
Then the termwise colimit $C^\bullet=\colim_i C_i^\bullet$ is canonically a complex of $R$-modules and $H^n(C^\bullet)=\colim_i H^n(C_i^\bullet)$ for all $n\in \bZ$,
since filtered colimit is exact.
%commutes with taking kernel and image.

Let $K^\bullet$ be a complex of finitely presented $R$-modules with finitely many nonzero terms.
Then there exists an index $i_0$ and a complex $K_{i_0}^\bullet$ of finitely presented $R_{i_0}$-modules with finitely many nonzero terms
such that $K_{i_0}^\bullet\otimes_{R_{i_0}} R=K^\bullet$ (cf. \citestacks{05N7}).
Note that in general $K_{i_0}\otimes_{R_{i_0}}^L R\neq K\in D(R)$.
Fix such a choice of $i_0$ and $K_{i_0}^\bullet$ and denote by $K_i^\bullet$ the tensor product $K_{i_0}^\bullet\otimes_{R_{i_0}} R_i$ for $i\geq i_0$,
so $K^\bullet=\colim_{i\geq i_0}K_i^\bullet$.
Then we have (see \citetwostacks{05DQ}{0G8P})
\begin{align*}
    \Hom_R(K^\bullet,C^\bullet)&=\Hom_{R_{i_0}}(K_{i_0}^\bullet,C^\bullet)\\
    &=\Hom_{R_{i_0}}(K_{i_0}^\bullet,\colim_{i\geq i_0}C_i^\bullet)\\
    &=\colim_{i\geq i_0}\Hom_{R_{i_0}}(K_{i_0}^\bullet,C_i^\bullet)\\
    &=\colim_{i\geq i_0}\Hom_{R_i}(K_i^\bullet,C_i^\bullet).
\end{align*}

Now, let $K\in D(R)$ be perfect \citestacks{0657}
and let $\psi:K\to C$ be a map in $D(R)$.
Then there exists a complex $K^\bullet$ of finite projective $R$-modules with finitely many nonzero terms and a map of complexes $\psi^\bullet:K^\bullet\to C^\bullet$ (see the proof of \citestacks{0658}) that represents $\psi$.
Fix a choice of $i_0$ and $K_{i_0}^\bullet$ as above,
and enlarge $i_0$ if necessary, we may assume that there exists a map of $R_{i_0}$-complexes $\psi_{i_0}:K_{i_0}^\bullet\to C_{i_0}^\bullet$ compatible with $\psi$.
Since being finite projective is the same as being a direct summand of a finite free module we see from \citestacks{05N7} that enlarging $i_0$ if necessary we may assume each term of $K_{i_0}^\bullet$ finite projective, in particular flat.
Then $K_{i_0}\otimes_{R_{i_0}}^L R= K\in D(R)$, and we have the following commutative diagram in $D(R)$:
\[
\begin{CD}
K_{i_0}\otimes_{R_{i_0}}^L R @>{\psi_{i_0}\otimes 1}>> C_{i_0}\otimes_{R_{i_0}}^L R\\
@V{\wr}VV @VVV\\
K@>{\psi}>> C
\end{CD}
\]
%Since $\psi$ is a quasi-isomorphism we see that $C_{i_0}\otimes_{R_{i_0}}^L R\to C$ is a split epimorphism in $D(R)$.
If, furthermore, we are given an integer $n$ and a cohomology class $\alpha_{i_0}\in H^n(C_{i_0}^\bullet)$
such that the image $\alpha\in H^n(C^\bullet)$ comes from a $\beta\in H^n(K^\bullet)$,
then $\beta$ comes from some $\beta_{i_0}\in H^n(K_{i_0}^\bullet)$ after possibly enlarging $i_0$.
The class $\alpha_{i_0}$ does not necessarily agree with $\psi_{i_0}(\beta_{i_0})$. % but it always does after possibly enlarging $i_0$.
However, since $H^n(C^\bullet)=\colim_i H^n(C^\bullet_i)$,
after possibly enlarging $i_0$ we will have $\alpha_{i_0}=\psi_{i_0}(\beta_{i_0})$.
The discussions above lead us to
\begin{Lem}\label{lem:EventuallySplitEpi}
Let $R=\colim_i R_i$ be a direct limit of rings.
Let $i_0$ be an index and $f_{i_0}:X_{i_0}\to \Spec(R_{i_0})$ be a quasi-compact and separated morphism of schemes.
Let $f_i:X_{i}\to \Spec(R_{i})$ and $f:X\to \Spec(R)$ be the base change of $f_{i_0}$ to $R_i$ and $R$ respectively.

Let $K\in D(R)$ be a perfect object and $\psi:K\to Rf_*\cO_X\in D(R)$ a map in $D(R)$.
Assume that there exists a class $\beta\in H^0(K)$ that maps to $1\in f_*\cO_X$.

%is perfect (for example, if $f$ is proper and $R$ is a regular Noetherian ring of finite Krull dimension).
Then for large enough $i$, $\psi$ factors through the canonical map $Rf_{i*}\cO_{X_i}\otimes_{R_i}^L R\to Rf_*\cO_X$,
in a way that $\beta\in H^0(K)$ is mapped to $1\in H^0(Rf_{i*}\cO_{X_i}\otimes_{R_i}^L R)$.
%is a split epimorphism in $D(R)$,
%and there exists a splitting that sends $1\in H^0(Rf_*\cO_X)$ to $1\in H^0(Rf_{i*}\cO_{X_i}\otimes_{R_i}^L R)$.
\end{Lem}
\begin{proof}
%When $f$ is proper and $R$ is a regular Noetherian ring of finite Krull dimension,
%$Rf_*\cO_X$ is perfect \citestacks{066Z},
%so the ``for example'' statement is justified.
Fix a finite affine cover $X_{i_0}=\bigcup_{j=1}^m U^{(j)}_{i_0}$.
Then the \v{C}ech complex $\check{C}_{i_0}^\bullet$ associated with this open cover and the structure sheaf $\cO_{X_{i_0}}$
computes $Rf_{i_0*}\cO_{X_{i_0}}$, see \citestacks{01XD}.
We can use the
%and the same holds for $X_i$ and $X$ using the
base change of this open cover to compute $Rf_{i*}\cO_{X_{i}}\ (i\geq i_0)$ and $Rf_{*}\cO_{X}$ with corresponding \v{C}ech complexes $\check{C}_{i}^\bullet=\check{C}_{i_0}^\bullet\otimes_{R_{i_0}}R_i$ and $\check{C}^\bullet=\check{C}_{i_0}^\bullet\otimes_{R_{i_0}}R=\colim_{i\geq i_0} \check{C}_{i}^\bullet$.
The canonical class $1\in H^0(\check{C}_{i}^\bullet)$ is defined by the unit element of the ring $\check{C}_{i}^0=\prod_{j=1}^m \cO_{X_i}(U^{(j)}_i)$.
Therefore we can apply the discussions above to the system $\check{C}_i^\bullet$
and for some $i\geq i_0$ find a commutative diagram
\[
\begin{CD}
K_{i}\otimes_{R_{i}}^L R @>{\psi_i\otimes 1}>> \check{C}_{i}\otimes_{R_{i}}^L R\\
@V{\wr}VV @VVV\\
K@>{\psi}>> \check{C}
\end{CD}
\]
in $D(R)$ and a class $\beta_{i}\in H^0(K_{i}^\bullet)$
that maps to $\beta\in H^0(K)$ and to $1\in H^0(\check{C}_{i})$ by $\psi_i$.
We win.
\end{proof}

\begin{Rem}
Using the language of derived $\infty$-categories \cite[1.3.2]{HA},
we can prove the lemma above as follows.
We may assume that $i_0$ is the smallest element of the index set.
The \v{C}ech complex consideration tells us
$Rf_*\cO_X=\operatorname{colim}_i Rf_{i*}\cO_{X_i}$
in the (bounded-above) derived $\infty$-category of abelian groups, and this identification preserves $R_i$-structure.
(This colimit is the ``homotopy colimit'' (cf. \citestacks{090Z}).)
Therefore
\[Rf_*\cO_X=\operatorname{colim}_{(i,j):i\leq j}\left( Rf_{i*}\cO_{X_i}\otimes_{R_i}^L R_j\right)\]
since $\{(i,i)\mid i\in\Lambda\}$ is a cofinal subset of  $\{(i,j)\mid i\leq j\}$.
Taking colimit with respect to $j$ we get
\[Rf_*\cO_X=\operatorname{colim}_{i} \left(Rf_{i*}\cO_{X_i}\otimes_{R_i}^L R\right)\]
in the (bounded-above) derived $\infty$-category of $R$-modules.
%When $Rf_*\cO_X$ is perfect, it is a
The perfect object $K$ is a
compact object in this $\infty$-category (cf. \cite[Propositions 7.2.4.2 and 7.1.1.15]{HA}),
so we obtain a factorization
$K\to Rf_{i*}\cO_{X_i}\otimes_{R_i}^L R$ for some $i$.
Enlarging $i$ we may assume $\beta$ is mapped to $1$ and we win.
\end{Rem}

\subsection{Direct limit of Noetherian birational derived splinters}
We turn to the main result of this section.

\begin{Thm}\label{thm:MDSlimitHARD}
%Let $S_0$ be a ring.
Let $(S_i)_{i}$ be a direct system of Noetherian rings with colimit $S$.
%Let $S_0$ be a finite $R_0$-algebra such that
%$S_0\otimes_{R_0} A$ is local for all local maps of local rings $S_0\to A$.
%$R_0\to A_0$ is surjective.
%Write $S_i=S_0\otimes_{R_0} R_i$ and $S=S_0\otimes_{R_0} R=\colim_i S_i$.
Assume that $S$ is Noetherian,
%that each $R_i$ is flat over $R_0$,
and that each $S_i$ is a birational derived splinter.
Then $S$ is a birational derived splinter.
\end{Thm}
\begin{proof}
Since $S$ and all $S_i$ are Noetherian, the problem is local by Corollary \ref{cor:MDSisLocal}. % and Lemma \ref{lem:LocalizeMDS}.
We may therefore assume that $S$ is local, that each $S_i$ is local, and that the transition maps are local.
%By assumption each $S_i$ is local.
% and $R_i\to R$ is faithfully flat, so each $R_i$ and thus $A_i$ is Noetherian.
By Lemma \ref{lem:characterizeNoetherianMDS} each $S_i$ is normal, hence an integral domain.
%This implies in particular $A$ reduced. % as $\Spec(A_i)\to\Spec(R_i)$ is finite and injective and $R_i$ is local.
%Thus $A$ is normal.
%We also note that $\dim R<\infty$.

Let $f:X\to \Spec(S)$ be an M-morphism.
By Corollary \ref{cor:Mlimitdomain},
there exists an index $i_0$ and an M-morphism  $f_{i_0}:X_{i_0}\to \Spec(S_{i_0})$
whose base change to $S$ gives $f.$
As usual, denote by $f_{i}:X_{i}\to \Spec(S_{i})$ the base change of $f_{i_0}$ to $\Spec(S_i)$ where $i\geq i_0$.
By Lemma \ref{lem:OneBmodif} each $f_i$ is an M-morphism and thus $S_i\to Rf_{i*}\cO_{X_i}$ splits in $D(S_i)$.
Therefore we obtain a map
$\epsilon_i: Rf_{i*}\cO_{X_i}\otimes_{S_i}^L S\to S$ in $D(S)$ that sends $1$ to $1$.

Now, let $K\in D(S)$ be perfect and let $S\xrightarrow{\varphi} K\xrightarrow{\psi} Rf_*\cO_X$ be a factorization of the canonical map $S\to Rf_*\cO_X$.
Since $S$ is Noetherian, $Rf_*\cO_X\in D(S)$ is pseudo-coherent (cf. \citestacks{08E8}).
Thus by Lemma \ref{lem:PerfToPseudoCoh} it suffices to show that $\varphi$ is a split mono in $D(S)$.

Denote by $\beta$ the image of $1\in S$ in $H^0(K)$, so $\psi$ maps $\beta$ to $1$.
%Note that $X_0$ is a proper $R_0$-scheme as $R_0\to S_0$ is finite.
By Lemma \ref{lem:EventuallySplitEpi}, for large $i$
there exists a map $\delta_i:K\to Rf_{i*}\cO_{X_i}\otimes_{S_i}^L S$ in $D(S)$ that sends $\beta$ to $1$.
Composing with $\epsilon_i$,
 we get a map $K\to S$
 in $D(S)$ that sends $\beta$ to $1$, in other words, a splitting of $\varphi$, as desired.
% The assumption that all $R_i$ are flat over $R_0$ implies
% $S_i\otimes_{R_i}^L R=S_0\otimes_{R_0}^L R_i\otimes_{R_i}^L R=S_0\otimes_{R_0}^L R=S$.
% Thus we have a map $Rf_*\cO_X\to S$
% in $D(R)$ that sends $1$ to $1$,
% which is a splitting of $S\to Rf_*\cO_X$ in $D(S)$ as desired.
\end{proof}

Combine with Proposition \ref{prop:charactrizeQBDS},
in equal characteristic zero we have the following result.

\begin{Cor}\label{cor:limitQmirror}
Let $A$ be a Noetherian quasi-excellent $\bQ$-algebra.
If $A$ is the direct limit of a system of Noetherian quasi-excellent $\bQ$-algebras that have rational singularities,
then $A$ has rational singularities.
\end{Cor}

In general, in view of Popescu's theorem \citestacks{07GC},
we have the following.

\begin{Cor}\label{cor:SmoothImpliesReg}
Let $S_0\to S$ be a regular map of Noetherian rings. %$B$.
If every smooth $S_0$-algebra is a birational derived splinter, then $S$ is a birational derived splinter.
\end{Cor}
%\begin{proof}
%Let $R_0\to S_0$ be a surjective map of Noetherian complete local rings where $R_0$ is regular.
%A regular local map between Noetherian complete local rings
%The map $S_0\to S$
%is formally smooth, see \citestacks{07PM}.
%By \cite[Th\'eor\`eme 19.7.2]{EGA04}, %(see also \citestacks{07NS}),
%there exists a flat local $R_0$-algebra $R$
%which is itself a Noetherian complete local ring
%such that $S_0\otimes_{R_0}R=S$.
%By \citestacks{07PM} the map $R_0\to R$ is regular, in particular $R$ is regular \citestacks{033A}.
%Immediate from
%\end{proof}
 
\begin{Rem}
Assume both $S_0$ and $S$ are complete local and the map $S_0\to S$ is local.
%The lemma above is used to reduce Theorem \ref{thm:SmCharp} to the case $B$ is a smooth $A$-algebra.
Let $k_0$ (resp. $k$) be the residue field of $S_0$ (resp. $S$).
If $k/k_0$ is separable \citestacks{030O}, we can argue as follows to prove Corollary \ref{cor:SmoothImpliesReg} avoiding previous derived category techniques.
By \citetwostacks{03C3}{05GH} we can find a residue extension (Definition \ref{def:ResidueExtension}) $S_0\to T$ that realizes $k_0\to k$.
$T$ is then formally smooth over $S_0$ by \citestacks{07PM}
and thus maps to $S$.
Now it is clear that $S$ is isomorphic to a power series algebra over $T$,
thus a residue extention of a localization of a polynomial algebra over $S_0$.
By Lemma \ref{lem:Gseparable=regular} this residue extension is regular.
Then we conclude by Theorem \ref{thm:ScalarFoverR}, which is proved independent of the methods in this section (modulo the proof of Corollary \ref{cor:IndEtMDS}).

However, in general $k/k_0$ is not necessarily separable.
A counterexample is given by
$S_0=k_0=\bF_p(t)$, $S=S_0[X]_{(X^p-t)}^\wedge$, so $k=\bF_p(t^{1/p})$.
\end{Rem}

\section{\'Etale ascent}\label{sec:Et}

We remark the following result on birational maps, which can be understood as a version of Chow's lemma for algebraic spaces \cite[Corollarie 5.7.13]{Raynaud}. % in our setting.
Nevertheless, for the reader's convenience, we (re-)do the proof using just schemes.
The arguments here are similar, but not identical, to
those in \cite[\S 5.7]{Raynaud}.
%, essentially repeating some arguments in \cite[\S 5.7]{Raynaud}
%essentially is  
\begin{Thm}%[cf. {\cite[Theorem A]{DTetale}}]
\label{thm:EtaleDominateM}
Let $A$ be a Noetherian reduced ring, % $K$ its total fraction ring, and assume $\dim K=0$. % (for example $A$ is Noetherian).
$A\to B$ an %inductively
\'etale ring map. % of reduced rings.
%Assume that $A$ has finitely many minimal primes
Let $g:Y\to T:=\Spec(B)$ be an M-morphism. % proper birational \citestacks{01RN} morphism.
Then there exists an M-morphism $f:X\to S:=\Spec(A)$ that admits a morphism $X\times_S T\to Y$ of $T$-schemes.

In particular, if $A$ is a birational derived splinter, so is $B$ (cf. {\cite[Theorem A]{DTetale}}).
\end{Thm}
\begin{proof}
The ``in particular'' statement follows from flat base change, see \S\S\ref{subsec:Derived}.
%and the fact reducedness ascends along a \'etale map \citestacks{033B}.

%Since $A\to B$ is \'etale, it is quasi-finite.
By Zariski's Main Theorem %\citestacks{0F2N},
there exists a factorization $T\xrightarrow{j} \overline{T}\xrightarrow{\pi} S$
where $\pi$ is finite % of finite presentation
and $j$ is an open immersion.
Write $\overline{T}=\Spec(\overline{B})$.
We can replace $\overline{T}$ by the scheme-theoretic image of $j$ to assume $j$ scheme-theoretically dominant, in particular, for the total fraction ring $K$ of $A$, $B\otimes_A K=\overline{B}\otimes_A K$, because a dense subset of the discrete space $\Spec(\overline{B}\otimes_A K)$ must be the whole space.

By Lemma \ref{lem:LocalizeProperFP} (which has a much shorter proof in the Noetherian case), there exists a proper morphism $\overline{g}:\overline{Y}\to \overline{T}$
such that $\overline{Y}\times_{\overline{T}}T=Y$.
We know $B\otimes_A K$ is the total fraction ring of $B=B_{\mathrm{red}}$ (cf. Lemma \ref{lem:Partial=TotalFraction}).
Since $g$ is an M-morphism, $g$ is an isomorphism after the base change $A\to K$. Since $\overline{B}\otimes_A K= B\otimes_A K$ the same is true for $\overline{g}$.
Picture:
\[
\begin{CD}
\overline{Y}\times_S \Spec(K)@>>>\overline{Y}\\
@V{\wr}VV  @VV{\overline{g}}V\\
\overline{T}\times_S \Spec(K)@>>>\overline{T}\\
@V{\mathrm{finite\ flat}}VV  @VVV\\
\Spec(K)@>>>S\\
\end{CD}
\]

By \citestacks{02KB} we know that a finite flat morphism of finite presentation is finite locally free (see \citestacks{02KA} for definition).
By \citestacks{06AC},
the composition $\overline{Y}\to \overline{T}\to S$ is finite locally free after the base change $\Spec(A_a)\to S$ for some nonzerodivisor $a\in A$.
By a consequence \citestacks{0B49} of ``flattening'' \cite[Th\'eor\`eme 5.2.2]{Raynaud},
there exists a blow-up $f:X\to S$ of an ideal supported on $\Spec(A/aA)$ set-theoretically and a closed subscheme $\overline{Y}'\subseteq \overline{Y}\times_S X$ finite locally free over $X$ that agrees with $\overline{Y}\times_S X$ over $\Spec(A_a)\subseteq S$.
The last bit follows from the definition of strict transform \citestacks{080D}.
We see that $f:X\to S$ is proper and that the base change of $f$ to $\Spec(A_a)$ is an isomorphism, so $f$ is an M-morphism. %, and that the closed immersion $\overline{Y}'\subseteq \overline{Y}\times_S X$ is of finite presentation (\citestacks{02FV}).
Picture:
\[
\begin{CD}
\overline{Y}'@>>>\overline{Y}\times_S X@>>>\overline{Y}\\
@.@VVV  @VV{\overline{g}}V\\
{}@.\overline{T}\times_S X@>>>\overline{T}\\
@.@VVV  @VV{\overline{\pi}}V\\
{}@.X@>{f}>>S\\
\end{CD}
\]

Let $Y'=\overline{Y}'\times_{\overline{T}} T$ and $h:Y'\to T\times_S X$ be the morphism induced by $\overline{g}$. Then the following diagram of schemes is commutative:
\[
\begin{CD}
Y'@>>>Y\\
@V{h}VV  @VV{g}V\\
T\times_S X@>>>T\\
@VVV  @VVV\\
X@>{f}>>S\\
\end{CD}
\]
It now suffice to show that $h$ is an isomorphism, so $T\times_S X$ dominates $Y$.

By our choice $Y'\to X$ is quasi-finite and flat.
Since $T\times_S X\to X$ is  \'etale, $h$ is quasi-finite and flat (cf. \citestacks{04R3}).
However, by construction, $h$ is proper, % of finite presentation.
so $h$ is finite flat, and we need to show that the degree of $h$ is $1$.
We have $\overline{Y}'=\overline{Y}\times_S X$ over $\Spec(A_a)$, so $Y'=Y\times_S X$ over $\Spec(A_a)$; also $Y=T$ after the base change $A\to K$ as noted before.
Thus $h$ is an isomorphism after the base change $A\to K$, so the locus where $\deg h=1$ contains the closure of the image of $b:T\times_S X\times_S\Spec(K)\to T\times_S X$.
However, $X\to S$ is a blowup, so a nonzerodivisor in $A$ is always a nonzerodivisor on $X$. Thus $X\times_S\Spec(K)\to X$ is (scheme-theoretically) dominant.
Then the same is true after the flat base change $T\to S$, so $b$ is dominant and $h$ is an isomorphism.
\end{proof}

\begin{Cor}%[cf. {\cite[Theorem 2.8]{Splinters}}]
\label{cor:IndEtMDS}
Let $A$ be a Noetherian birational derived splinter and let $B$ be a direct limit of \'etale $A$-algebras.
If $B$ is Noetherian, then $B$ is a birational derived splinter.
\end{Cor}
\begin{proof}
Immediate from Theorem \ref{thm:MDSlimitHARD}.
Alternatively, one can apply Corollaries \ref{cor:MDSisLocal} and \ref{cor:Mlimitdomain} and flat base change (\S\S \ref{subsec:Derived}); details of this route are left to the reader.
%By Lemma \ref{lem:characterizeNoetherianMDS}, $A$ is normal, hence so is $B$.
%%By Corollary \ref{cor:MDSisLocal} (or Lemma \ref{lem:ProdMDS})
%we may assume $B$ a normal domain.
%Write $B=\colim_i B_i$ where each $B_i$ is an \'etale $A$-algebra, hence an MDS by Theorem \ref{thm:EtaleDominateM}.
%Note that each $B_i$ is a finite product of normal domains, so we may replace $\Spec(B_i)$ by its only connected component that contains the image of $\Spec(B)$ to assume that $B_i$ are (normal) domains.
%Note that the transition maps $B_i\to B_j$ are flat (cf. \citestacks{00U7}), so $B$ is flat over each $B_i$.
%A flat map of integral domains is injective.
%We conclude by Lemma \ref{lem:MDSlimit}.
\end{proof}

\begin{Rem}
We can also deduce from Theorem \ref{thm:EtaleDominateM} that if $A\to B$ is a local map of Noetherian local rings that is essentially \'etale,
and $A$ is pseudo-rational in the sense of \cite{LT81},
then so is $B$.
\end{Rem}

%\begin{Rem}
%With more effort, we can weaken the Noetherian assumption in Theorem \ref{thm:EtaleDominateM} to that the total fraction ring $K$ of $A$ satisfies $\dim K=0$, and in  Corollary \ref{cor:IndEtMDS} to that both $A$ and $B$ finitely many minimal primes, and possibly weaker.
%We do not present the arguments here as the results are not needed.
%\end{Rem}

\section{Regular residue extensions}\label{sec:ResidueExtn}

We continue to use our ad-hoc definition of ultraproduct of local rings \cite[Definition 3.1]{Splinters}.

\begin{Def}
Let $X$ be an index set and $A_x\ (x\in X)$ be a family of local rings.
An \emph{ultraproduct} of the rings $A_x\ (x\in X)$ is the localization of the ring $\prod_x A_x$ at a maximal ideal.

An \emph{ultrapower} of a local ring $A$ is an ultraproduct of a family $A_x\ (x\in X)$ where each $A_x=A$.
\end{Def}

For convenience, we state the next two lemmas for ultrapowers.
One can easily extend them to general ultrapoducts.

\begin{Lem}\label{lem:DomainUltraprod}
Let $A$ be a local domain. Then any ultrapower of $A$ is an integral domain.
\end{Lem}
\begin{proof}
Let $A_\natural= (A^X)_\fM$ be an ultrapower of $A$ where $X$ is an index set.
To show $A_\natural$ an integral domain, it suffices to show that if $a=(a_x)_{x\in X}$ and $b=(b_x)_{x\in X}$ are two elements in $A^X$ such that $ab=0$, then there exists an $s\not\in \fM$ such that $sa=0$ or $sb=0$.

Let $U=\{x\in X\mid a_x=0\}$ and $V=\{x\in X\mid b_x=0\}$, so $U\cup V=X$ as $A$ is an integral domain.
Let $e_U\in A^X$ be the idempotent defined by $(e_U)_x=1$ when $x\in U$ and $(e_U)_x=0$ when $x\not\in U$ and define $e_V$ similarly.
Then $(1-e_U)(1-e_V)=0$, so either $e_U\not\in \fM$ or $e_V\not\in \fM$.
But $e_Ua=0=e_Vb$.
We win.
\end{proof}

\begin{Lem}%[cf. {\cite[Lemma 3.2]{Splinters}}]
\label{lem:MDSUltraprod}
Let $S$ be a local MDS, $S_\natural$ an ultrapower of $S$.
%Let $S$ be a finite and finitely presented $A$-algebra that is an MDS.
If $S$ has finitely many minimal primes,
then $S_\natural$ is an MDS.
\end{Lem}
\begin{proof}
%$S$ is a finitely presented $A$-module, see \citestacks{0564}.
%We can write $A_\natural=(A^X)_\fM$, and $S\otimes_A A^X=S^X$ by \citestacks{059K},
%so
$S_\natural$ is a localization of some $S^X$ which is an MDS by Lemma \ref{lem:ProdMDS}.

Let $K$ be the total fraction ring of $S$.
Note that $S$ is reduced by Lemma \ref{lem:MDSnormal}.
By Lemma \ref{lem:TotalFractionProduct}, $K^X$ is the total fraction ring of $S^X$ and $\dim (K^X)=0.$
The result now follows from Lemma \ref{lem:LocalizeMDS}.
\end{proof}

%\begin{Lem}\label{lem:ultrapowerFlat}
%Let $A$ be a Noetherian local ring.
%Then any ultrapower of $A$ is a faithfully flat $A$-algebra.
%\end{Lem}
%\begin{proof}
%Note that an ultrapower $A_\natural=(A^X)_{\fM}$ of $A$ is canonically an $A$-algebra via the diagonal map.
%$A_\natural$ is a flat $A$-algebra by %citetwostacks{05CZ}{05CY}.
%Faithful flatness now follows from Lemma \ref{lem:mcontainedinM}.
%we just need to show that the maximal ideal $\fM$ of $A^X$ lies above the maximal ideal $\fm$ of $A$.
%But this is because $\fm$ is contained in the Jacobson radical of $A^X$, see \citestacks{0AME}.
%\end{proof}

\begin{Thm}%[cf. {\cite[Theorem 1.4]{Splinters}}]
\label{thm:ScalarFoverRMIRROR}
Let $(S,\fm)\to (S',\fm')$ be a regular homomorphism of Noetherian local rings with $\fm'=\fm S'$.
If $S$ is a birational derived splinter, so is $S'.$
In particular, the completion of a local G-ring that is a birational derived splinter is a birational derived splinter.
\end{Thm}
\begin{proof}
%Recall that for a Noetherian ring, being a birational derived splinter is the same as being an MDS, see Lemma \ref{lem:characterizeNoetherianMDS}.
%
The proof is mostly verbatim to the proof of \cite[Theorem 1.4]{Splinters},
replacing the ingredients
Lemma 2.2, Theorem 2.8, and Lemma 3.2 there by
our Proposition \ref{prop:PureDescendMDS}, Corollary \ref{cor:IndEtMDS}, and Lemma \ref{lem:MDSUltraprod} respectively.
The only difference occurs at the last step,
where for Proposition \ref{prop:PureDescendMDS} to apply,
we need to know that the pure map $S'\to S_\natural$ sends nonzerodivisors to nonzerodivisors.
However, $S$ is an integral domain (Lemma \ref{lem:characterizeNoetherianMDS}), hence so is $S_\natural$ (Lemma \ref{lem:DomainUltraprod}).
We win.
\end{proof}

\begin{Cor}%[cf. {\cite[Theorem C]{DTetale}}]
\label{cor:completionMDS}
Let $S$ be a Noetherian local G-ring that is a birational derived splinter.
Then any separable residue extension (Definition  \ref{def:ResidueExtension}) of $S$ is a birational derived splinter.
%In particular, the completion $S^\wedge$ is a birational derived splinter.
\end{Cor}
\begin{proof}
In view of Lemma \ref{lem:Gseparable=regular}, this is the special case of Theorem \ref{thm:ScalarFoverRMIRROR}
where $S$ is a G-ring.
\end{proof}
%by the same reason Corollary \ref{cor:completionMDS} follows from Theorem \ref{thm:ScalarFoverR}.
%\begin{Cor}\label{cor:SplinterScalarG}
%Let $(S,\fm_S,k_S)\to (A,\fm_A,k_A)$ be a flat homomorphism of Noetherian local rings.
%Assume that $\fm_A=\fm_SA$, that $k_A/k_S$ is separable,
%and that $S$ is a G-ring.
%If $S$ is a splinter, so is $A$.
%\end{Cor}

\section{Openness in prime characteristic and regular ascent}\label{sec:OpenandReg}
%In this section we develop parallel results of \cite{DTopen} in this section and then prove Theorem \ref{thm:SmCharp}.
We use the notaions in \cite[\S\S 2.3]{DTopen} (Frobenius pushforward), and in our \S\S\ref{subsec:Derived} (trace ideal).

We need one preparation before we go into the argument.
\begin{Lem}\label{lem:Aquasiexcellent}
Let $A$ be as in Theorem \ref{thm:open}.
Then $A$ is quasi-excellent. Moreover, the normal, Cohen-Macaulay, and $F$-pure loci of $\Spec(A)$ are open.
\end{Lem}
\begin{proof}
%For a proof of the fact $A$ quasi-excellent, see the proof of \cite[Theorem 4.3.1]{DTopen}, third paragraph.
By a theorem of Kunz, $F$-finite Noetherian rings are excellent, see \cite[Theorem 108]{Matsumura}.
Noetherian local G-rings are quasi-excellent,
and the property of being quasi-excellent is preserved by extensions essentially of finite type,
see \cite[(34.A)]{Matsumura}.
We got quasi-excellence.

%$F$-finite Noetherian rings have finite Krull dimension by \cite[Proposition 1.1]{Kun76}, and so do Noetherian local rings.
%Therefore $\dim A<\infty$.
%
%A Noetherian ring of dimension $\leq d$ is Cohen-Macaulay if and only if it satisfies Serre's condition $(S_d)$.
%Therefore
Since $A$ is quasi-excellent, \cite[Proposition 6.11.8]{EGA4_2} shows that
the function $\fp\mapsto \dim A_\fp-\operatorname{depth} A_\fp$ is constructible and upper semi-continuous on $\Spec(A)$.
Thus the Cohen-Macaulay locus of $A$ is open.
The normal locus is open as well by, for example, \emph{ibid}., Corollaire 6.13.5.

Finally, openness of $F$-pure locus is trivial in the $F$-finite case (as pure is equivalent to split, see \citestacks{058L}) and is \cite[Corollary 3.5]{MurGamma} in the other case.
\end{proof}

\begin{Def}[={\cite[Definition 3.1.1]{DTopen}}]
An ideal $\fa$ of an $\bF_p$-algebra $A$ is \emph{uniformly $F$-compatible} if for all $e\in\bZ_{>0}$ and all $A$-linear maps
$\varphi:F_*^e A\to A$, we have $\varphi(F_*^e \fa)\subseteq\fa$.
%
%We say $\fa$ \emph{flat-universally uniformly $F$-compatible} if for all flat ring maps $A\to A'$,
%$\fa A'$ is uniformly $F$-compatible.
\end{Def}

\begin{Lem}[cf. {\cite[Lemma 3.2.3]{DTopen}}]
\label{lem:TraceIdealFcomp}
Let $A$ be an $\bF_p$-algebra and $X$ a qcqs $A$-scheme.
Then the ideal $\ft(X/A)$ is uniformly $F$-compatible.
%If $A$ is Noetherian and $X$ is proper over $A$,
%then $\ft(X/A)$ is %flat-universally
%uniformly $F$-compatible.
\end{Lem}
\begin{proof}
%By flat base change, \S\S\ref{subsec:Derived},
%%it suffices to show $\fa$ uniformly $F$-compatible.
%
Let $f:X\to \Spec(A)$ be the structural map and consider the following commutative diagram of schemes
\[
\begin{CD}
X@>{F^e}>> X\\
@V{f}VV       @V{f}VV\\
\Spec(A)@>{F^e}>>\Spec(A)
\end{CD}
\]
where $F^e$ is the $e$-th iterated absolute Frobenius of $X$ or $\Spec(A)$.
The map $\cO_X\to F^e_*\cO_X$ induces a map $Rf_*\cO_X\to Rf_*F^e_*\cO_X=F^e_*Rf_*\cO_X$,
the identity because $F^e$ is affine.
If $\eta:Rf_*\cO_X\to A$ is a map in $D(A)$, and $\varphi:F^e_*A\to A$ is an $A$-linear map,
then the composition
\[
\eta':Rf_*\cO_X\to Rf_*F^e_*\cO_X=F^e_*Rf_*\cO_X\xrightarrow{F^e_*\eta} F^e_*A\xrightarrow{\varphi} A
\]
is also a map in $D(A)$.
If the image of $\eta$ in $A$ is $a$ then it is clear that the image of $\eta'$ in $A$ is $\varphi(F^e_*a)$, so $\varphi(F_*^e \ft(X/A))\subseteq\ft(X/A)$ as desired.
\end{proof}

%We use the notations in Subsection \ref{subsec:Derived} in the proof of the theorem below.
\begin{Thm}[cf. {\cite[Theorem 4.3.1]{DTopen}}]
\label{thm:TraceIdealsFinitelyMany}
Let $A$ be as in Theorem \ref{thm:openoverFpure}, or be as in Theorem \ref{thm:open} and $F$-pure.

Then the set of ideals $\Sigma_A:=\{\ft(X/A)\mid X\mathrm{\ proper\ over\ }A\}$ is finite.
\end{Thm}
\begin{proof}
If $A$ is $F$-finite, then $A$ is $F$-split since $A$ is Noetherian and $F$-pure.
Therefore the number of uniformly $F$-compatible ideals of $A$ is finite, see \cite[Proposition 3.4.1]{DTopen}, so we conclude by Lemma \ref{lem:TraceIdealFcomp} that $\Sigma_A$ is finite.
If $A$ is essentially of finite type over a Noetherian local G-ring, then since $A$ is $F$-pure, there exists a faithfully flat ring map $A\to B$ where $B$ is Noetherian, $F$-finite, and $F$-pure by \cite[Theorem 3.4]{MurGamma} (cf. the proof of \cite[Corollary 3.5]{MurGamma}) and $\Sigma_B$ is finite by what we just proved.
For $\ft(X/A)\in\Sigma_A$, we have $\ft(X/A)B=\ft(X_B/B)\in\Sigma_B$ (see \S\S\ref{subsec:Derived}; 
this is why we need properness), so the finiteness of $\Sigma_B$ implies that of $\Sigma_A$.

Finally, when $A$ is as in Theorem \ref{thm:openoverFpure},
we can find
a faithfully flat ring map $R\to R'$ where $R'$ is Noetherian, $F$-finite, and $F$-pure, see \cite[Lemma 2.13]{Splinters}.
By \cite[Proposition 2.4]{Has10} $R'\otimes_R A$ is $F$-pure,
and is Noetherian and $F$-finite
since it is essentially of finite type over $R'$.
%as in the proof of \cite[Theorem 2.14]{Splinters}, and
The rest of the argument above carries verbatim.
\end{proof}

\begin{Cor}\label{cor:normallocusoverFpure}
Let $A$ be as in Theorem \ref{thm:openoverFpure}.
Then the normal (resp. splinter) locus of $A$ is open.
\end{Cor}
\begin{proof}
By \cite[Proposition 2.12]{Splinters},
openness of the splinter locus
follows from Theorem \ref{thm:TraceIdealsFinitelyMany}.
The proof of the openness of the normal locus follows the same lines, but let us do it below for the reader's convenience.

For every $\fp\in\Spec(A)$ such that $A_\fp$ is normal, since $A$ is Noetherian, there exists $f\in A,f\not\in\fp$ such that $A_f$ is an integral domain.
We may thus assume from the outset that $A$ is an integral domain.

For a Noetherian integral domain $C$ put $$\Sigma_\nu(C)=\{\ft(D/C)\mid D\textrm{\ is\ a\ finite\ extension\ of\ }C\textrm{\ in\ the\ fraction\ field\ of\ }C\}.$$
It is elementary to see that $\Sigma_\nu(C_P)=\{IC_P\mid I\in\Sigma_\nu(C)\}$ for all $P\in\Spec(C)$.
Also, if $C\subseteq D$ is an extension of integral domains of the same fraction field, then $C\to D$ splits as a map of $C$-modules if and only if $C=D$.
One sees that the normal locus of a Noetherian integral domain $C$ is always $\Spec(C)\setminus\bigcup_{\fa\in \Sigma_\nu(C)}V(\fa)$.

Now, by Theorem \ref{thm:TraceIdealsFinitelyMany}, $\Sigma_\nu(A)\subseteq\Sigma_A$ is finite.
This concludes the proof.
%For every
\end{proof}

\begin{Cor}\label{cor:normalizationofFpure}
Let $A$ be as in Theorem \ref{thm:openoverFpure}.
Then the normalization of $A$ is finite over $A$.
\end{Cor}
\begin{proof}
By \cite[Proposition 2.4]{Has10} $A$ is $F$-pure, in particular reduced.
Thus by Corollary \ref{cor:normallocusoverFpure} $A_f$ is normal for some nonzerodivisor $f\in A$.
Moreover, for all $\fp\in \Spec(A)$,
the completion $A^\wedge_\fp$ is $F$-pure (see for example \cite[Lemma 3.26]{Has10}), thus reduced,
so the normalization of $A_\fp$ is finite \citestacks{032Y}.
Our result now follows from \cite[Proposition 6.13.6]{EGA4_2}, which is stated for domains but works for reduced rings as well.
%once we replace the reference \citestacks{0332} by.
\end{proof}

\begin{Rem}[cf. {\cite[Example 4.0.3]{DTetale}}]\label{rem:Openfail}
The previous corollaries (and thus Theorem \ref{thm:TraceIdealsFinitelyMany}) fail for a general $F$-pure Noetherian ring, and in fact,
fails even for a general $F$-pure Noetherian G-ring.

For a counterexample, we let $k$ be an algebraically closed field of characteristic $p$,
and let $(A_0,\fm_0)$ be a $k$-algebra essentially of finite type over $k$ that is a non-normal $F$-pure integral domain.
For instance we can take $A_0=k[x,y]_{(x,y)}/(xy+x^3+y^3)$,
so $A_0^\wedge\cong k[[x,y]]/(xy)$.
Then by \cite{Hoc73},
%an infinite tensor power of $A_0$ over $k$ localized at a suitable multiplicative subset
there is a Noetherian integral domain $A$ which has local rings at maximal ideals of the form $A_0\otimes_k K$ where $K/k$ is a field extension (thus is $F$-pure and a G-ring),
and has non-open normal locus,
and thus the normalization of $A$ is not finite \cite[Proposition 6.13.2]{EGA4_2}.
Since a non-normal ring is neither a splinter nor a birational derived splinter, \cite{Hoc73} also shows
that such $A$ constructed has non-open splinter and birational derived splinter locus.
\end{Rem}

Now we prove our main theorem about openness of the birational derived splinter locus.
It is similar to the proof of Corollary \ref{cor:normallocusoverFpure} and \cite[Theorem 2.14]{Splinters},
except for extra care of normality due to the assumptions in Lemma \ref{lem:LocalizeM}. % holds

\begin{Thm}%[cf. {\cite[Theorem 4.3.1]{DTopen}}]
\label{thm:openMIRROR}
The followings hold.
\begin{enumerate}
    \item\label{FpureOpen} Let $A$ be as in Theorem \ref{thm:openoverFpure}.
    Then the locus of prime ideals $\fp$ of $A$ such that $A_\fp$ is a birational derived splinter is open.
    
    \item\label{Gopen} Let $A$ be as in Theorem \ref{thm:open}.
    Then the locus of prime ideals $\fp$ of $A$ such that $A_\fp$ is an $F$-pure birational derived splinter is open.
\end{enumerate}
% and  are true.
\end{Thm}

\begin{proof}
In case (\ref{Gopen}), we may assume $A$ $F$-pure by Lemma \ref{lem:Aquasiexcellent}.
Thus in both cases we need to show the locus
\[
\{\fp\in\Spec(A)\mid A_\fp\textrm{\ is\ a\ birational\ derived\ splinter}\}
\]
is open. %, since being a birational derived splinter is the same as being an MDS for a Noetherian ring, see Lemma \ref{lem:characterizeNoetherianMDS}.
A Noetherian birational derived splinter is normal by Lemma \ref{lem:characterizeNoetherianMDS}, so we may assume $A$ normal by either Lemma \ref{lem:Aquasiexcellent} or Corollary \ref{cor:normallocusoverFpure}.

We have shown that $\Sigma_A$  is finite (Theorem \ref{thm:TraceIdealsFinitelyMany}), in particular,\[\Sigma':=\{\ft(f)\mid f\textrm{\ is\ an\ M-morphism\ to\ }\Spec(A)\}\] is finite.
For every $\fp\in\Spec(A)$ and every M-morphism $g:Y\to\Spec(A_\fp)$, $g$ extends to an M-morphism $f:X\to\Spec(A)$, see Lemma \ref{lem:LocalizeM}, and $\ft(f)A_\fp=\ft(g)$, see \S\S\ref{subsec:Derived}.
On the other hand, a localization of an M-morphism is an M-morphism (Lemma \ref{lem:Mbasechange}; note that our ring $A$ is reduced).
We conclude that the MDS locus is
\begin{align*}
     &\{\fp\in\Spec(A)\mid \fa A_\fp=A_{\fp},\forall \fa\in \Sigma'\}\\
    =&\bigcap_{\fa\in \Sigma'}\left(\Spec(A)\setminus V(\fa)\right)
\end{align*}
which is open since $\Sigma'$ is finite.
\end{proof}
%We deduce Theorem \ref{thm:SmCharp} as a corollary with an argument similar to the proof of .

\begin{Thm}
\label{thm:SmCharpMIRROR}
Let $S\to A$ be a regular homomorphism of Noetherian $\bF_p$-algebras.
If $S$ is an $F$-pure birational derived splinter, so is $A$.
\end{Thm}
\begin{proof}
Note that $F$-purity localizes and ascends along regular maps, see \cite[Proposition 2.4]{Has10}, so in the (mostly implicit) d\'evissage process below $S$ is always $F$-pure and so is $A$.
%Details omitted.
%We also note that for Noetherian rings, birational derived splinter and MDS are the same, see Lemma \ref{lem:characterizeNoetherianMDS}.

We only need to consider the special case where $A$ a smooth $S$-algebra, see Corollary \ref{cor:SmoothImpliesReg}.
Now, the proof of {\cite[Theorem 1.1]{Splinters}} applies here, once we replace the ingredients
Theorem 2.9, Corollary 2.6, Lemma 2.2, Theorem 2.14, and Theorem 1.3 there
by our Corollary \ref{cor:IndEtMDS}, Corollary \ref{cor:MDSisLocal}, Proposition \ref{prop:PureDescendMDS}, Theorem \ref{thm:openMIRROR}(\ref{FpureOpen}), and Theorem \ref{thm:ScalarFoverRMIRROR} respectively.
\end{proof}

\begin{Rem}
In fact, we only used Theorem \ref{thm:ScalarFoverRMIRROR}
in the case
$S'=S[Y]_{\fm S[Y]}$.
In this case, the map $S\to S'$ actually satisfies the statement of Theorem \ref{thm:EtaleDominateM}.
This is because an M-morphism $Z\to \Spec(S')$ is dominated by a blowup of an ideal $J\subseteq S'$; and if $J=IS'$ for an ideal $I\subseteq S[Y]$,
then the base change of the blowup of $\Spec(S)$ of the ideal generated by the coefficients of elements of $I$ to $S'$ dominates $Z$.
This construction appears in \cite{LT81}
immediately before its Theorem 2.1; details omitted.
\end{Rem}

\section{Questions}\label{sec:Q}

%We ask if we can remove $F$-purity, and even remove the positive characteristic $p$ Theorem \ref{thm:open} is true.
\begin{Ques}
Let $A$ be essentially of finite type over a local G-ring.
Is the birational derived splinter locus of $A$ always open?
What about a general quasi-excellent ring, or an excellent ring?
\end{Ques}
In equal characteristic zero,
the birational derived splinter locus of a quasi-excellent ring is always open, as easily follows from Proposition \ref{prop:charactrizeQBDS}.
%this holds by the existence of resolutions.
%Similarly, in positive or mixed characteristic,
%if resolutions exist,
%it is not clear to the author if that will imply openness.
%Grauert-Riemanschneider,

%The following two questions are closely related.
\begin{Ques}
Let $S\to A$ be a regular homomorphism of Noetherian rings.
If $S$ is a birational derived splinter, is $A$ necessarily a birational derived splinter?
\end{Ques}
In equal characteristic zero,
if both $S$ and $A$ are quasi-excellent,
this is an easy consequence of Proposition \ref{prop:charactrizeQBDS} again.

We can ask if a strengthening holds,
see Question \ref{Ques:FlatAscent} below.

\begin{Ques}\label{Ques:CartierLift}
Let $(A,\fm)$ be a Noetherian local ring, $t\in \fm$ a nonzerodivisor.
If $A/tA$ is a birational derived splinter, is $A$ necessarily a birational derived splinter?
\end{Ques}

\begin{Ques}\label{Ques:FlatAscent}
Let $S\to A$ be a flat homomorphism of Noetherian rings.
If $S$ is a birational derived splinter, and the fibers of $S\to A$ are geometrically birational derived splinters, is $A$ necessarily a birational derived splinter?
\end{Ques}

These two questions are closely related.
If we replace ``birational derived splinter'' by ``splinter'' or ``derived splinter,''
then negative answers are given
in \cite{Singh},
since $F$-regularity implies splinter (see \cite[Theorem 3.5]{MP} and \cite{LS99}) and thus derived splinter by \cite{Bha12}.
However, the statements for birational derived splinter may be true,
and they hold for related singularity types: rational singularity in chacteristic zero \cite{Elk78},
$F$-rationality \cite[Theorems 5.1 and 7.12]{MP},
and BCM-rationality (\cite[Proposition 3.4]{MS18}, for Question \ref{Ques:CartierLift}).

We shall mention that if both Questions \ref{Ques:CartierLift} and \ref{Ques:FlatAscent} have affirmative answers,
then we will have further results about fibers of a map of Noetherian rings geometrically being birational derived splinters.
See \cite{Mur22}.

\end{document}